\documentclass{article}
\usepackage{amssymb, amsthm, amsmath, amsfonts}
\usepackage{graphicx}
\usepackage[T2A]{fontenc}
\usepackage{float}
\usepackage{subfigure}
\usepackage{enumitem}

\newtheorem{theorem}{Theorem}

\newtheorem{proposition}{Proposition}
\newtheorem{remn}{Remark}
\newtheorem{lem}{Lemma}
\begin{document}
\sloppy
\title{Morse-Smale 3-diffeomorphisms with saddles of the same unstable manifold dimension}
\date{}
\author{E.M. Osenkov, O.V. Pochinka}

\maketitle

\begin{abstract}
In this paper, we consider a class of Morse-Smale diffeomorphisms defined on a closed 3-manifold (non-necessarily orientable) under the assumption that all their saddle points have the same dimension of the unstable manifolds. The simplest example of such diffeomorphisms is the well-known ``source-sink'' or ``north pole - south pole'' diffeomorphism, whose non-wandering set consists of exactly one source and one sink. Such systems, as Reeb showed back in 1946, can be realized only on the sphere. We generalize his result, namely, we show that diffeomorphisms from the considered class also can be defined only on the 3-sphere.
\end{abstract}



\section*{INTRODUCTION AND FORMULATION OF THE RESULTS}

The class of dynamical systems, introduced by S. Smale in 1960 \cite{smale1960morse} and known today as Morse-Smale systems, played not the least role in the formation of the modern dynamical systems theory. The study of these systems remains an important part of it because they form a class of structurally stable systems which, in addition, have zero topological entropy \cite{PALIS1969385}, \cite{palis2000structural}, \cite{robinson1998dynamical}, that makes them in this sense by ``the simplest'' structural stable systems. 

A close relation of the Morse-Smale diffeomorphisms ($MS$\!-diffeomorphisms) with the topology of the ambient manifold allows us to realize various topological effects in the dynamics of such systems. The classical example demonstrating such a relation is systems with exactly two points of extreme Morse indices. In this case, it follows from Reeb's theorem \cite{reeb},  the ambient manifold is homeomorphic to the sphere.

Another brilliant illustration of researched relations is the decomposition of an orientable 3-manifold into a connected sum of $\mathbb S^2 \times \mathbb S^1$ whose  number of summands is completely determined by a structure of the non-wandering set of an MS-diffeomorphism without heteroclinic curves defined on it. This result was obtained in papers by H. Bonatti, V. Grines, and V. Medvedev \cite{Bonatti2002}, \cite{Grines2003} and is based on the breakthrough result about the existence of a tame neighborhood of a 2-sphere with one point of wildness. The ideas that authors put into their proofs have been extremely helpful in our research.

The present paper is a straightforward generalization of Reeb's Theorem on the following class of difeomorphisms. 
Let $f$ be an $MS$\!-diffeomorphism defined on a closed connected 3-manifold $M^3$ and all its saddle points have the same dimension of their unstable manifolds. Denote this class as $\mathcal G.$ Then we can formulate the main result of this work.
	
\begin{theorem}\label{result} Any closed connected $3$-manifold $M^3$, admitting a diffeomorphism $f \in \mathcal G$, is homeomorphic to the 3-sphere.
\end{theorem}
{\bf Acknowledgement.} The work was supported by the Russian Science Foundation (Project No. 23-71-30008).

\section{AUXILIARY INFORMATION AND FACTS}

This section introduces basic concepts and facts from topology and dynamical systems theory.

\subsection{Some topological facts}

Let $X$, $Y$ be topological spaces, $A \subset X$ and $B \subset Y$ are their subsets and $g:A \to B$ is a homeomorphism. Let $\sim$ be the minimal equivalence relation on $X\sqcup Y$ for which $a\sim g(a)$ for all $a\in A$. The factor space for this equivalence relation is said to
be obtained by gluing the space $Y$ to the space $X$ by the map $g$, written $X \cup_g Y$.

Let $X$, $Y$ be compact $n$-manifolds,  $D_1 \subset X$, $D_2 \subset Y$ be subsets homeomorphic to $\mathbb D^n$,  $h_1:\mathbb D^n \to D_1$, $h_2:\mathbb D^n \to D_2$ be the corresponding homeomorphisms and  $g: \partial D_1 \to \partial D_2$ be a homeomorphism such that the map $\left. h^{-1}_2 g h_1 \right|_{\partial \mathbb D^n} : \mathbb S^{n-1} \to \mathbb S^{n-1}$ reverses orientation. Then the space $X \# Y = (X \setminus \mathrm{int}\, D_1) \cup_g (Y\setminus \mathrm{int}\, D_2)$ is called {\it the connected sum of $X$ and $Y$.}

If $X \subset Y$ then the map $i_X:X \to Y$ such that $i_X(x) = x$ for all $x \in X$, is called {\it the inclusion map of X into Y}.

Let $X$ and $Y$ be $C^r$-manifolds. Denote by $C^r(X,Y)$ the set of all $C^r$-maps $\lambda:X\to Y$. A map $\lambda:X\to Y$ is said to be {a $C^r$-embedding} if it is a $C^r$-diffeomorphism onto the subspace $\lambda(X)$.
$C^0$-embedding is also called {\it a topological embedding}.

A topological embedding $\lambda: X \to Y$ of an $m$-manifold $X$ into an $n$-manifold $Y$ $(m \leq n)$ is said to be {\it locally flat at the point $\lambda(x)$, $x \in X$,} if the point $\lambda(x)$ is in the domain of such a chart $(U,\psi)$ of the manifold $Y$ that $\psi(U \cap \lambda(X)) = \mathbb R^m$, here $\mathbb R^m \subset \mathbb R^n$ is the set of points for which the last $n-m$ coordinates equal to $0$ or $\psi(U \cap \lambda(X)) = \mathbb R^m_+$, here $\mathbb R^m_+ \subset \mathbb R^m$ is the set of points with the non-negative last coordinate.

An embedding $\lambda$ is said to be {\it tame} and the manifold $X$ is said to be {\it tamely embedded} if $\lambda$ is locally flat at every point $x \in X$. Otherwise the embedding $\lambda$ is said to be {\it wild} and the manifold $X$ is said to be {\it wildly embedded}. A point $\lambda(x)$ which is not locally flat, is said to be {\it the point of wildness.}

\begin{proposition}[\cite{Bonatti2000}, Theorem 4]\label{T2}
    Let $T$ be a two-dimensional torus tamely embedded in the manifold $\mathbb S^2 \times \mathbb S^1$ in such a way that $i_{T*}(\pi_1(T))\neq 0$. Then $T$ bounds in $\mathbb S^2 \times \mathbb S^1$ a solid torus.
\end{proposition}

\begin{proposition}[\cite{Bonatti2002}, Proposition 0.1]\label{NoWS} Let $M^3$ be a closed connected $3$-manifold and let $\lambda: \mathbb S^2 \to M^3$ be a topological embedding of the $2$-sphere which is smooth everywhere except one point. Let $\Sigma = \lambda(\mathbb S^2)$. Then any neighborhood $V$ of the sphere $\Sigma$ contains a neighborhood $K$ diffeomorphic to $\mathbb S^2 \times [0,1]$.\footnote{This fact was proven in \cite{Bonatti2002} for an orientable manifold $M^3$, but the proof don't use the orientability anywhere. So we can use this result in our case too.}
\end{proposition}

Let us assume that the empty set and only the empty set {\it  has a dimension $-1$} $(\dim \varnothing = -1)$. The separable metric space $X$ {\it has a dimension $\leq n$ $(\dim X \leq n)$} if any neighborhood $V_p$ of a point $p \in X$ contains a neighborhood $U_p$ such that $\dim (\partial U_p) \leq n-1$. The space $X$ {\it has a dimension $n$ $(\dim X = n)$} if the statement $\dim X \leq n$ is true and the statement $\dim X \leq n-1$ is false. 

It is said that a subset $D$ of a connected space $X$ {\it divides it} if the space $X\setminus D$ is disconnected.
\begin{proposition}[\cite{HurWall_1948}, Corollary 1, p.48]\label{n-2}
    Any connected $n$-manifold cannot be divided by a subset of the  dimension $\leq n-2$.
\end{proposition}

\subsection{Morse-Smale diffeomorphisms}
Here and below, we assume that $M^n$ is a closed connected 3-manifold with a metric $d$ and a map $f:M^n \to M^n$ is a diffeomorphism. 

{\it The trajectory} or {\it the orbit} of a point $x \in M^n$ is the set $\mathcal O_x = \{f^m(x), m\in \mathbb Z\}$. 

A set $A \subset M^n$ is said to be {\it $f$-invariant} if $f(A) = A$, that is $A$ consists of hole orbits.

A compact $f$-invariant set $A \subset M^n$ is called {\it an attractor} of the diffeomorphism $f$ if it has a compact neighborhood $U_A$ such that $f(U_A) \subset \mathrm{int}\, U_A$ and $A = \bigcap\limits_{k\geq 0}f^k(U_A)$. The neighborhood $U_A$ in this case is said to be {\it trapping}. {The basin of the attractor $A$} is the set $$W^s_A=\{x\in M^n:\, \lim\limits_{k\to+\infty}d(f^k(x),A)=0\}.$$

{\it A repeller} and its basin are defined as an attractor and its basin for $f^{-1}.$

A point $x \in M^n$ is said to be {\it a wandering} for the diffeomorphism $f$ if there is an open neighborhood $U_x$ of x such that $f^k(U_x) \cap U_x = \varnothing$ for all $k \in \mathbb N$. Otherwise, the point $x$ is said to be {\it a non-wandering}. The set of all non-wandering points is called the {\it non-wandering set} and it will be denoted by $\Omega_f$. The non-wandering set $\Omega_f$ is $f$-invariant and if $\Omega_f$ is finite, then it consists only of {periodic} points, i.e. such points $p \in M^n$ that there exists the natural number $m$ for which $f^m(p)=p$. If this equality is not satisfied for any natural number $k < m$, then $m$ is called {\it the period of a point $p$}, denote it by $m_p$. 

For a periodic point $p$, let us define sets $$W^s_p=\{x\in M^n:\lim\limits_{k\to+\infty}d(f^{km_p}(x),p)=0\}$$ and
	$$W^u_p=\{x\in M^n:\lim\limits_{k\to-\infty}d(f^{km_p}(x),p)=0\},$$
which are called, respectively, {\it stable} and {\it unstable manifolds} of the point $p$. These sets are also known as {\it invariant manifolds} of the point $p$.

A periodic point $p$ with a period $m_p$ is said to be {\it hyperbolic} if the absolute values of each eigenvalues of the Jacobi matrix $\left.\left(\dfrac{\partial f^{m_p}}{\partial x}\right)\right\vert_p$ is not equal to $1$. If the absolute values of all the eigenvalues are less than $1$, then $p$ is called an {\it attracting}, a {\it sink point} or a {\it sink}; if the absolute values of all the eigenvalues are greater than 1, then p is called a {\it repelling}, a {\it source point} or a {\it sink}. Attracting and repelling fixed points are called {\it nodes}. A hyperbolic periodic  point which is not a node is called a {\it saddle point} or a {\it saddle}.

The hyperbolic structure of the  periodic point $p$ and the finiteness of the non-wandering set implies that its the stable and the unstable manifolds are smooth submanifolds of $M^n$ which are diffeomorphic to $\mathbb R^{q_p}$ and $\mathbb R^{n-q_p}$ respectively, where $q_p$ is a {\it Morse index of $p$}, that is the number of the eigenvalues of Jacobi matrix whose the absolute value is greater than $1$. 

A connected component $\ell^u_p\,(\ell^s_p)$ of the set $W^u_p \setminus p\, (W^s_p \setminus p)$ is called an unstable (stable) {\it separatrix} of the periodic point $p$. For $p$ let $\nu_p$ be $+1 (-1)$ if $\left. f^{m_p}\right\vert_{W^u_p}$ preserves (reverses) orientation and let $\mu_p$ be $+1 (-1)$ if $\left.f^{m_p}\right\vert_{W^s_p}$ preserves (reverses) orientation. 

A diffeomorphism $f: M^3 \to M^3$ is called a {\it Morse-Smale diffeomorphism} $(f \in MS(M^3))$ if 

1) the non-wandering set $\Omega_f$ is finite and hyperbolic;

2) for every two distinct periodic points $p,q$ the manifolds $W^s_p, W^u_q$ intersect transversally.

Note that all the facts below are proved in the case when $M^n$ is orientable, but the direct check allows us to verify the correctness of these results for non-orientable manifolds as well.

\begin{proposition}[\cite{GrMePo2016}, Theorem 2.1]\label{th_2.1.1}
Let $f \in MS(M^3)$. Then  
\begin{enumerate}
    \item $M^{n} = \underset{p\in\Omega_{f}}{\bigcup}W^{u}_{p}$,
    \item $W^{u}_{p}$ is a smooth submanifold of the manifold $M^n$  diffeomorphic to $\mathbb R^{q_p}$ for every periodic point $p \in \Omega_f$,
    \item $\mathrm{cl}(\ell^{u}_{p})\setminus(\ell^{u}_{p}\cup p) = \underset{r\in \Omega_f:\ell^{u}_{p}\cap W^{s}_{r}\neq \varnothing}{\bigcup}W^{u}_{r}$ for every unstable separatrix $\ell^{u}_{p}$ of a periodic point $p\in \Omega_f.$
\end{enumerate}
\end{proposition}

If $\sigma_1, \sigma_2$ are distinct periodic saddle points of a diffeomorphism $f \in MS(M^n)$ then the intersection  $W^s_{\sigma_1} \cap W^u_{\sigma_2} \neq \varnothing$ is called a {\it heteroclinic}. If $\dim(W^s_{\sigma_1} \cap W^u_{\sigma_2}) > 0$ then a connected component of the intersection $W^s_{\sigma_1} \cap W^u_{\sigma_2}$ is called a {\it heteroclinic manifold} and if $\dim(W^s_{\sigma_1} \cap W^u_{\sigma_2}) = 1$ then it is called a {\it heteroclinic curve}. If $\dim(W^s_{\sigma_1} \cap W^u_{\sigma_2}) = 0$ then the intersection $W^s_{\sigma_1} \cap W^u_{\sigma_2}$ is countable, each point of this set is called a {\it heteroclinic point}  and the orbit of a heteroclinic point is called the {\it heteroclinic orbit}. 

\begin{proposition}[\cite{GrMePo2016}, Proposition 2.3]\label{sep_sph}
    Let $f\in MS(M^n)$ and  $\sigma$ be a saddle point of $f$ such that the unstable separatrix $\ell_\sigma^u$ has no heteroclinic intersections. Then $$\mathrm{cl}(\ell^u_\sigma) \setminus (\ell_\sigma^u \cup \sigma) = \{\omega\},$$  where $\omega$ is a sink point. If $q_\sigma = 1$ then $\mathrm{cl}(\ell^u_\sigma)$ is an arc topologically embedded into $M^n$ and if $q_\sigma \geq 2$ then $\mathrm{cl}(\ell^u_\sigma)$ is the sphere $\mathbb S^{q_\sigma}$ topologically embedded into $M^n$.
\end{proposition}

A diffeomorphism $f \in MS(M^n)$ is called a {\it ``source-sink''} diffeomorphism if its non-wandering set consists of a unique sink and a unique source.

\begin{proposition}[\cite{GrMePo2016}, Theorem 2.5]\label{NS_diff}
    If a diffeomorphism $f\in MS(M^n)$, $n >1$, has no saddle points  then
    $f$ is a ``source-sink'' diffeomorphism and the manifold $M^n$ is homeomorphic to the $n$-sphere $\mathbb S^n$.\footnote{The second part of this statement can be known as a special case of the Reeb's Theorem \cite{reeb}}
\end{proposition}

\begin{proposition}[\cite{Grines2010}, Theorem 1]\label{AfR} 
    Let $f \in MS(M^n)$ and  $\Omega_A$ be such a subset of $\Omega_f$ that the set $$A = \Omega_0 \cup W^u_{\Omega_A}$$ is closed and $f$-invariant. Then 
    \begin{enumerate}
        \item the set $A$ is an attractor of the diffeomorphism $f$;
        \item $W^s_A = \bigcup\limits_{p \in (A \cap \Omega_f)} W^s_p$;
        \item $\dim A = \max\limits_{p \in (A \cap \Omega_f)}\{q_p\}$.
     \end{enumerate}
\end{proposition}

For an orbit $\mathcal O_p$ of a point $p$, let $m_{\mathcal O_p}=m_p$, $q_{\mathcal O_p}=q_p$, $\nu_{\mathcal O_p}=\nu_p$, $\mu_{\mathcal O_p} = \mu_p$,  $W_{\mathcal{O}_p}^s = \bigcup\limits_{q\in\mathcal{O}_p}W^s_q$,  $W_{\mathcal {O}_p}^u = \bigcup\limits_{q\in\mathcal{O}_p}W^u_q$.

Following the classic paper by S. Smale \cite{smale1967differentiable}, we introduce on the set of periodic orbits of $f \in MS(M^n)$ a partial order $\prec$: $$\mathcal{O}_i \prec \mathcal{O}_j \iff W_{\mathcal{O}_i}^s \cap W_{\mathcal{O}_j}^u \neq \varnothing.$$

According to Szpilrajn's theorem \cite{Szpilrajn1930}, any partial order (including the Smale order) can be extended to a total order. Let us consider a special kind of such total order on the set of all periodic orbits.

We say that numbering of the periodic orbits $\mathcal{O}_1, \cdots, \mathcal{O}_{k_f}$ of the diffeomorphism $f\in MS(M^n)$ is a {\it dynamical} if it satisfies the following conditions:
\begin{enumerate}
\item $\mathcal{O}_i \prec \mathcal{O}_j \implies i \leqslant j$;
\item $q_{\mathcal{O}_i} < q_{\mathcal{O}_j} \implies i<j$.
\end{enumerate}

\begin{proposition}[\cite{GrMePo2016}, Proposition 2.6]
    For any diffeomorphism $f \in MS(M^n)$ there is a dynamical numbering of its periodic orbits.
\end{proposition}

\subsection{Orbit spaces}

In this section, we present concepts and facts whose detailed presentation and proof can be found in the monograph \cite{GrMePo2016}.

Let $f:M^n \to M^n$ be a diffeomorphism and let $X \subset M^n$ be an $f$-invariant set.  It can be checked directly that the relation $x \sim y \iff \exists k \in \mathbb{Z}: y = f^k(x)$ is an equivalence relation on $X$. The quotient set $X/f$ induced by this relation is called an {\it  orbits space of the action of $f$ on $X$}. Let us denote by $p_{_{X/f}}:X\to X/f$  the natural projection.
{\it A fundumental domain of the action of $f$ on $X$} is a  closed set $D_X \subset X$ such that there is a set $\tilde D_X$ satisfying:
\begin{enumerate}
    \item $\mathrm{cl}(\tilde D_X) = D_X$;
    \item $f^k(\tilde D_X) \cap \tilde D_X = \varnothing$ for each $k \in \mathbb Z \setminus \{0\}$;
    \item $\bigcup\limits_{k \in \mathbb Z} f^k(\tilde D_X) = X$.
\end{enumerate}

If the projection $p_{_{X/f}}$ is a cover and the orbits space $X/f$ is connected then, by virtue of the Monodromy Theorem (see, for example, \cite{GrMePo2016}, p.60), for a loop $\hat c \subset X/f$, closed at a point $\hat x$, there exists its lift $c \subset X$ which is a path joining  points $x \in p^{-1}_{_{X/f}}(\hat x)$ and $f^k(x)$. In this case, a map $\eta_{_{X/f}}: \pi_1(X/f) \to \mathbb{Z}$, defined by the formula $\eta_{_{X/f}}([\hat c]) = k$, is a homomorphism which is called {\it induced by the cover $p_{_{X/f}}$}. 

\begin{proposition}\label{cycle} Let $f$ and $f'$ be diffeomorphisms defined on $f$- and $f'$-invariant set $X$. If $\hat h:X/f \to X/f'$ is a homeomorphism for which $\eta_{X/f} = \eta_{X/f'}\hat h$ then there is a  homeomorphism $h:X \to X$  which is a lift of $\hat h$  $(p_{_{X/f'}} h=\hat h p_{_{X/f}})$ and such that $hf = f'h$.
\end{proposition}

\begin{proposition}[\cite{GrMePo2016}, Theorem 2.1.3]\label{T_2.1.3.}
Let $f \in MS(M^n)$, $A$ be an attractor of $f$,   $\ell^s_A=W^s_A\setminus A$, $\hat\ell^s_A=\ell^s_A/f$ and  $D_{\ell^s_A}$ be a fundamental domain of the action of $f$ on $\ell^s_A$. Then the projection $p_{\hat\ell^s_A}$ is a cover and the orbits space $\hat\ell^s_A$ is a smooth closed $n$-manifold homeomorphic to $p_{\hat\ell^s_A}(D_{\ell^s_A})$. In particular, if the attractor $A$ coincides with a sink orbit then the manifold $\hat \ell^s_A$ is homeomorphic to following manifolds:
\begin{itemize}
    \item $\mathbb S^1$  for $n=1$;
    \item $\mathbb{S}^{n-1}\tilde{\times} \mathbb{S}^{1}$  for $n>1,\,\nu_A=-1$;
    \item $\mathbb{S}^{n-1}{\times} \mathbb{S}^{1}$  for $n>1,\,\nu_A=+1$.
\end{itemize}
\end{proposition}

\section{TOPOLOGY OF $3$-MANIFOLDS ADMITTING DIFFEOMORPHISMS FROM THE CLASS $\mathcal G$}

Recall that $\mathcal G$ is a class of Morse-Smale diffeomorphisms $f:M^3 \to M^3$ defined on a closed connected $3$-manifold $M^3$ (not necessarily orientable), with non-wandering set $\Omega_f$ whose all saddle points have the same dimension of their unstable manifolds. 

This section is focused on the proof of the main result of this paper.

{\bf Theorem 1. } {\it Any closed connected $3$-manifold $M^3$, admitting a diffeomorphism $f \in \mathcal G$, is homeomorphic to the $3$-sphere.}

To prove the main result let us state some auxiliary facts.

\begin{remn}
    Further, without loss of generality, up to the power of the diffeomorphism, one may assume that $\Omega_f$ consists of fixed points only and for all $p \in \Omega_f$ the numbers $\nu_p$ and $\mu_p$ equal to $+1$. Moreover, for definiteness, we suppose  that the set $\Omega_1$ is empty.
\end{remn}

\begin{lem}\label{omega} For any diffeomorphism $f\in \mathcal G$, the set $\Omega_0$ consists of a unique sink.
\end{lem}

\begin{proof} Let $$R = W^s_{\Omega_2} \cup \Omega_3.$$ By virtue of Statement \ref{AfR}, the set $R$ is a repeller of the diffeomorphism $f$ and $\dim R=1$. It follows from  Statement \ref{n-2} that $M^3 \setminus R$ is connected. On the other hand, according to  Statement \ref{th_2.1.1}, $M^3\setminus R = W_{\Omega_0}^s$. From the above we conclude that the set $\Omega_0$ consists of a unique sink.
\end{proof}

Let us denote by $\omega$ the unique sink of the diffeomorphism $f\in \mathcal G$.

\begin{lem}\label{sigma} In the non-wandering set of any diffeomorphism $f \in \mathcal G$ there exists a saddle $\sigma$ such that $\ell_\sigma^u \subset \ell_\omega^s$.
\end{lem}

\begin{proof} By  Lemma \ref{omega}, the fixed points of the diffeomorphism $f$ admit the following dynamical order:
\begin{equation}\label{order}
    \begin{gathered}
    \omega \prec \sigma_1 \prec \cdots \prec \sigma_k \prec \alpha_1 \prec \cdots \prec \alpha_s,\\
    \text{where } \Omega_2 = \{ \sigma_1, \cdots, \sigma_k \},\, \Omega_3 = \{ \alpha_1, \cdots, \alpha_s \}.
    \end{gathered}
\end{equation}
Assume that $\sigma = \sigma_1$. Then, it follows from order (\ref{order}) that $$\forall~ p \in \Omega_f\setminus \omega \implies \ell_\sigma^u \cap W_p^s = \varnothing.$$
In the other words, $\ell_\sigma^u$ can only intersect with $W_\omega^s$. By Statement \ref{th_2.1.1} (1), any point $x \in \ell_\sigma^u$ has to lie on the stable manifold of some fixed point. Hence, $\ell_\sigma^u \subset \ell_\omega^s$.
\end{proof}

Further, let the saddle $\sigma \in \Omega_2$ satisfies the conclusion of Lemma \ref{sigma}, and let $\Sigma_\sigma = \mathrm{cl}(\ell_\sigma^u)$. It follows from Statements \ref{sep_sph} and \ref{th_2.1.1} (2) that $\Sigma_\sigma = \ell_\sigma^u \cup \{\omega\} \cup \{\sigma\}$ is an embedding of the two-dimensional sphere (see Fig.  \ref{fig:Sigma}). This embedding is smooth everywhere except, maybe, the point $\omega$. Let $\mathcal M_\sigma = M^3 \setminus \Sigma_\sigma$. 
\begin{figure}[!ht]
\begin{center}
\includegraphics[width=0.3\textwidth]{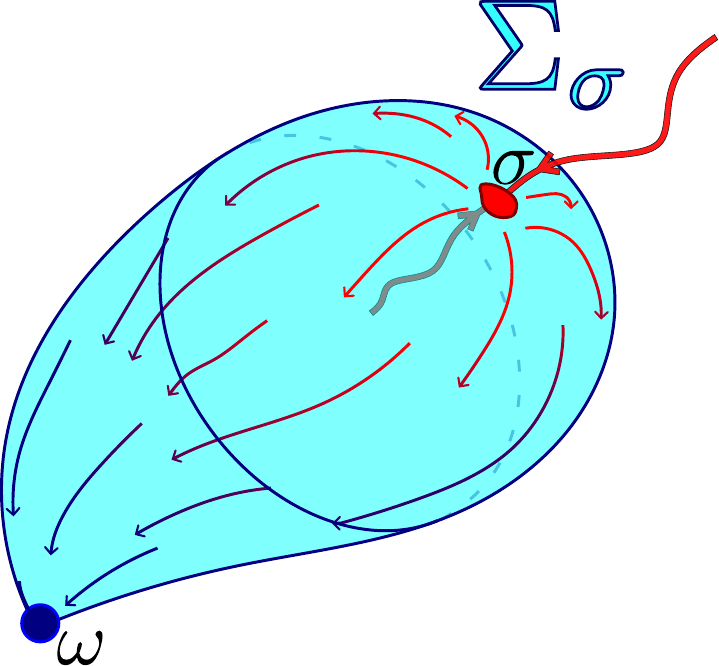}
    \caption{Sphere $\Sigma_\sigma$}\label{fig:Sigma}
\end{center}
\end{figure}

\begin{lem}\label{msi} The manifold $\mathcal{M}_\sigma$ is disconnected.
\end{lem}
\begin{proof} Since for any manifold the notions of connectivity and path connectivity are equivalent, they will be used interchangeably hereafter.

{\bf Step 1. } First of all, let us proof that the set $\mathcal L_\sigma = \ell_\omega^s \setminus \ell_\sigma^u$ is disconnected. Suppose the contrary: any two distinct points $x,y \in \mathcal L_\sigma$ can be connected by a path in $\mathcal L_\sigma$ (see Fig. \ref{fig:Ligma}).
\begin{figure}[!ht]
    \begin{minipage}{0.45\textwidth}
        \centering
\includegraphics[width=\textwidth]{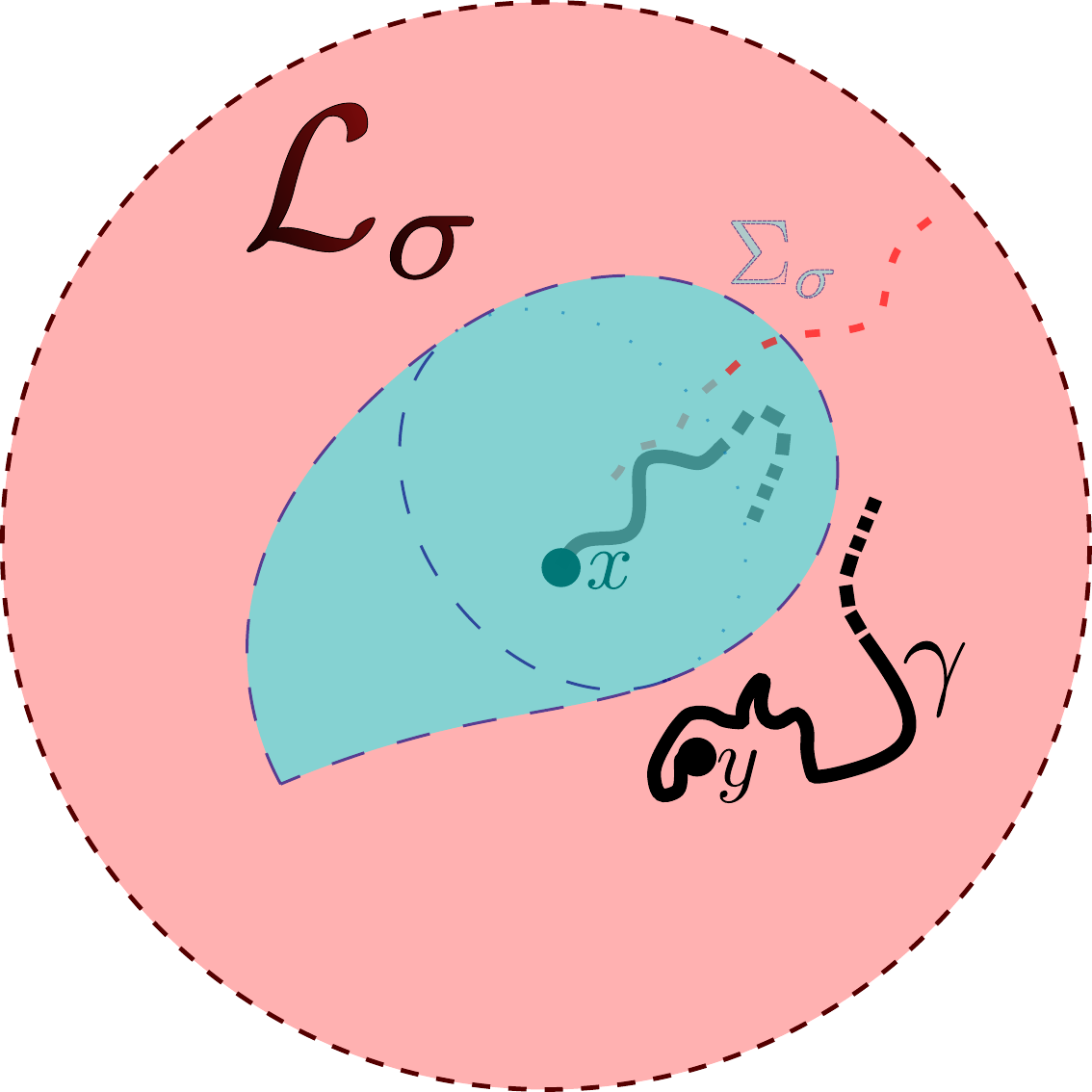}
        \caption{The space $\mathcal L_\sigma.$}\label{fig:Ligma}
    \end{minipage}
    \hfill
    \begin{minipage}{0.45\textwidth}
	\centering
\includegraphics[width=\textwidth]{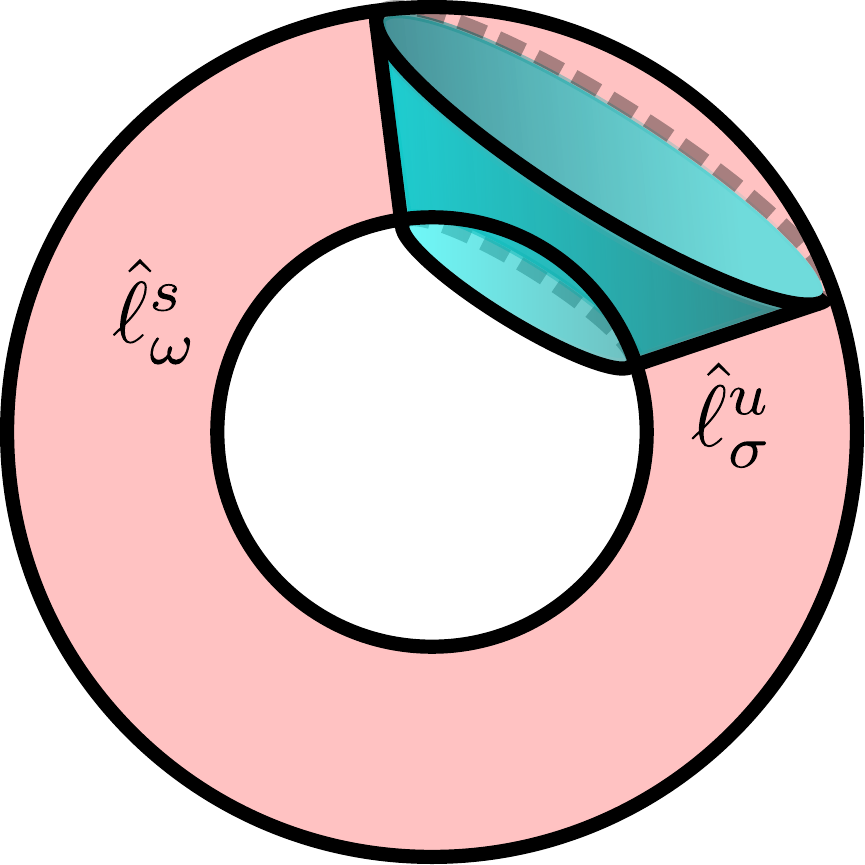}
	\caption{The orbits space of the sink basin}\label{fig:orbit}
    \end{minipage}
\end{figure}

Consider the orbits space $\hat \ell_\omega^s = \ell_\omega^s/f$ of the sink $\omega$ and put $p_\omega = p_{\hat \ell_\omega^s}:\ell_\omega^s \to \hat \ell_\omega^s$, $\eta_\omega = \eta_{\hat \ell_\omega^s}: \pi_1(\hat \ell_\omega^s) \to \mathbb Z$. By  Statement \ref{T_2.1.3.}, the map $p_\omega$ is a cover,  $\hat \ell_\omega^s$ is homeomorphic to $\mathbb S^2 \times \mathbb S^1$ and $\hat \ell_\sigma^u$ is homeomorphic to the two-dimensional torus (see Fig. \ref{fig:orbit}).

Since $\ell_\sigma^u \subset \ell_\omega^s$, then $\hat \ell_\sigma^u \subset \hat \ell_\omega^s$. Moreover, $\hat \ell_\sigma^u = p_\omega(\ell_\sigma^u)$ that, by Statement \ref{th_2.1.1},   implies that $\hat \ell_\sigma^u$ is a smooth embedding of the 2-torus into $\hat \ell_\omega^s$ (see Fig. \ref{fig:orbit}). By  Statement \ref{T_2.1.3.}, homomorphism $\eta_\omega$ is non-trivial and it follows from its definition that $i_*(\pi_1(\hat \ell_\sigma^u)) \neq 0$. 

Then, using  Statement \ref{T2}, one may conclude that $\hat \ell_\sigma^u$ bounds in $\hat \ell_\omega^s$ a solid torus and, consequently, it divides this orbits space into two connected components. Let us choose a point in each component and denote them $\hat x$ and $\hat y$. From their pre-images we take two points $x \in p_\omega^{-1}(\hat x)$ and $y \in p_\omega^{-1}(\hat y)$. Since we assumed that $\mathcal{L}_\sigma$ is path-connected then there exists a path $\gamma :[0,1] \to \mathcal L_\sigma : \gamma(0) = x,~ \gamma(1) = y$. Then, by continuity of $p_\omega$, the map $\hat \gamma = p_\omega\gamma: [0,1] \to \hat \ell_\omega^s \setminus \hat \ell_\sigma^u$ is a path between $\hat x$ and $\hat y$ in $\hat \ell_\omega^s \setminus \hat \ell_\sigma^u$, that is a contradiction.

Thus, $\mathcal L_\sigma$ is disconnected.

{\bf Step 2. } Let us prove that $\mathcal M_\sigma = M^3 \setminus \Sigma_\sigma$ is not connected. Suppose the contrary: it is connected. Let us note that $\dim\,\mathcal M_\sigma=3$, because it is an open subset of the manifold $M^3$. Then, by  Statement \ref{n-2}, $\mathcal M_\sigma\setminus R$ is  connected. On the other hand, $\mathcal M_\sigma\setminus R=(M^3 \setminus \Sigma_\sigma) \setminus R = W^s_\omega \setminus \Sigma_\sigma=\mathcal L_\sigma$ and it contradicts the conclusion of the previous step. So, $\mathcal M_\sigma$ is disconnected.
\end{proof}

Let us introduce a diffeomorphism $a:\mathbb R^3 \to \mathbb R^3$ by the rule $a(x,y,z) = (\frac{x}{2},\frac{y}{2},\frac{z}{2})$. It has a unique non-wandering point, a sink $O(0,0,0)$. Let $\ell = \mathbb R^3 \setminus O$. 

As well as before, let $f \in \mathcal G$, $\sigma$ satisfies the conclusion of Lemma \ref{sigma} and $\Sigma_\sigma = \mathrm{cl}(\ell_\sigma^u)$. By  Statement \ref{AfR}, sphere ${\Sigma_\sigma}$ is an attractor of diffeomorphism $f$ with the basin $W_{\Sigma_\sigma}^s = W_\sigma^s \cup W_\omega^s$. Let $\ell_{\Sigma_\sigma}^s = W_{\Sigma_\sigma}^s \setminus \Sigma_\sigma$.

\begin{lem}\label{lel} The manifold $\ell_{\Sigma_\sigma}^s$ consists of two connected components $\ell_1$, $\ell_2$, and for each of the components $\ell_i$ there exists a diffeomorphism $h_i: \ell_i \to \ell$, conjugating $\left. f \right|_{\ell_i}$ with $\left. a \right|_{\ell}$.
\end{lem}

\begin{proof}  By virtue of  Statement \ref{T_2.1.3.}, the orbits space $\hat{\ell}_{\Sigma_{\sigma}}^s = \ell_{\Sigma_{\sigma}^s}/f$ is a smooth closed $3$-manifold.
Let us prove that $\hat \ell^s_{\Sigma_\sigma} \cong \mathbb S^2 \times \mathbb S^1 \sqcup \mathbb S^2 \times \mathbb S^1$.  

By Statement \ref{NoWS}, the attractor $\Sigma_\sigma$ has a neighborhood $K_\sigma \subset W_{\Sigma_\sigma}^s$ diffeomorphic to $\mathbb S^2 \times [0,1]$ (see Fig. \ref{fig:Kigma}). 
    \begin{figure}[!ht]
        \centering
        \includegraphics[width = 0.4\textwidth]{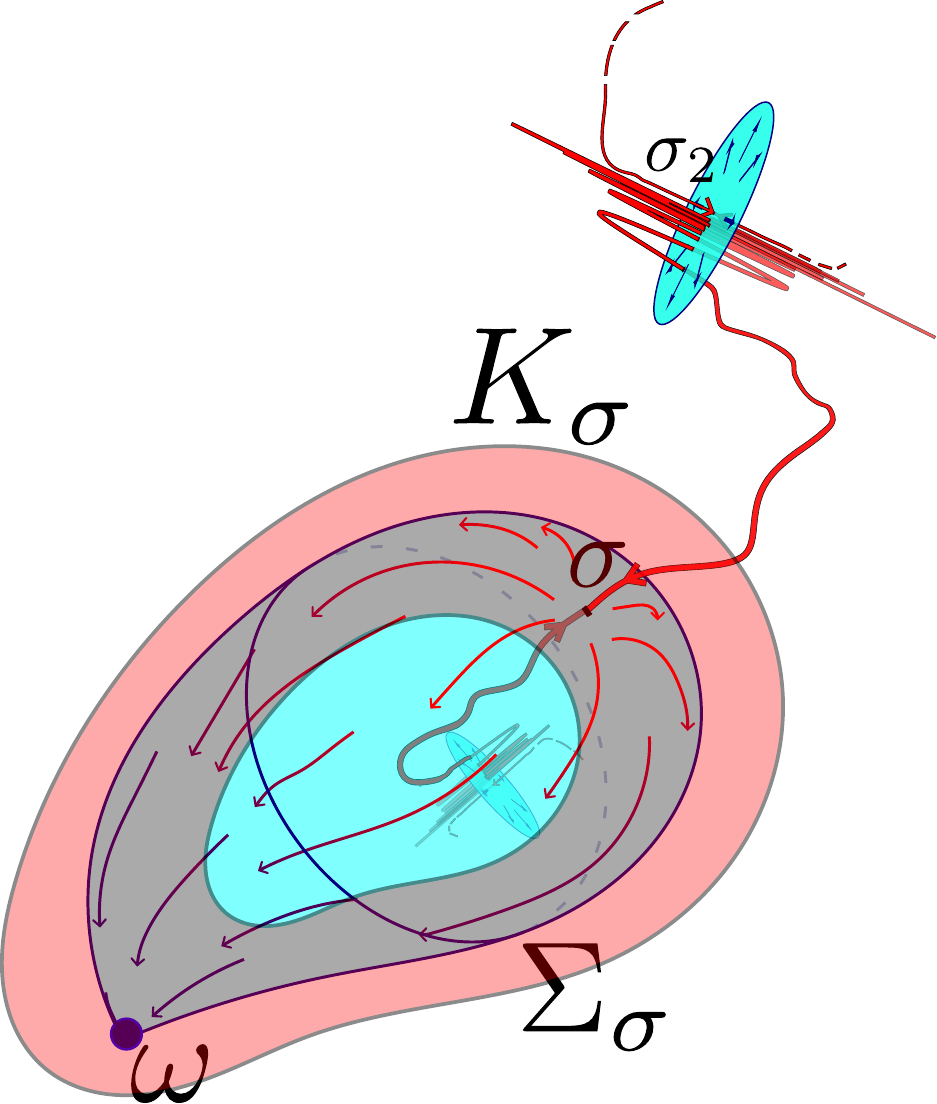}
        \caption{The neighbohood $K_\sigma \cong \mathbb S^2 \times [0,1]$ of the sphere $\Sigma_\sigma$}
        \label{fig:Kigma}
    \end{figure}
    Let us show that there exists a natural number $N$ such that $f^N(x) \in \mathrm{int}\, K_\sigma$ for any $x \in K_\sigma$. Since $\partial K_\sigma \subset W^s_{\Sigma_\sigma}$, then for all $x\in \partial K$ there exist such a closed neighborhood $U_x \subset \partial K_\sigma$ and a natural number $\nu_x$ that for any $\nu \geq \nu_x$ it is true that $f^\nu(U_x) \subset \mathrm{int} K_\sigma$. Due to the compactness of $\partial K_\sigma$, there exists a finite subcover of $\partial K_\sigma$ in $\{U_x, x \in \partial K_\sigma \}$. Thus, one may choose the desired number $N$ as the maximum of numbers $\nu_x$ corresponding to the neighbourhoods of $U_x$ in the chosen subcover. Without loss of generality, we assume the number $N$ to be $1$, then $f(K_{\Sigma_\sigma}) \subset \mathrm{int} \, K_\sigma$  (see Fig. \ref{fig:attraction}). It follows from  Lemma \ref{msi} that the sphere $\Sigma_\sigma$ separates in $K_\sigma$ the connected components of its boundary. Whence, according to [\cite{Grines2003}, Theorem 3.3], $K_{\sigma}\setminus \mathrm{int} \, f(K_{\sigma}) \cong \mathbb{S}^2 \times [0,1] \sqcup\mathbb{S}^2 \times [0,1]$. It follows from the construction that the manifold $K_{\sigma}\setminus \mathrm{int} \, f(K_{\sigma})$ is a fundamental domain of the action of $f$ on the space $\ell_{\Sigma_\sigma}^s$. Then by  Statement \ref{T_2.1.3.}, $\hat \ell_{\Sigma_\sigma}^s \cong \mathbb{S}^2 \times \mathbb S^1 \sqcup \mathbb{S}^2 \times \mathbb S^1$.
\begin{figure}[!ht]
    \centering
    \includegraphics[width = 0.55\textwidth]{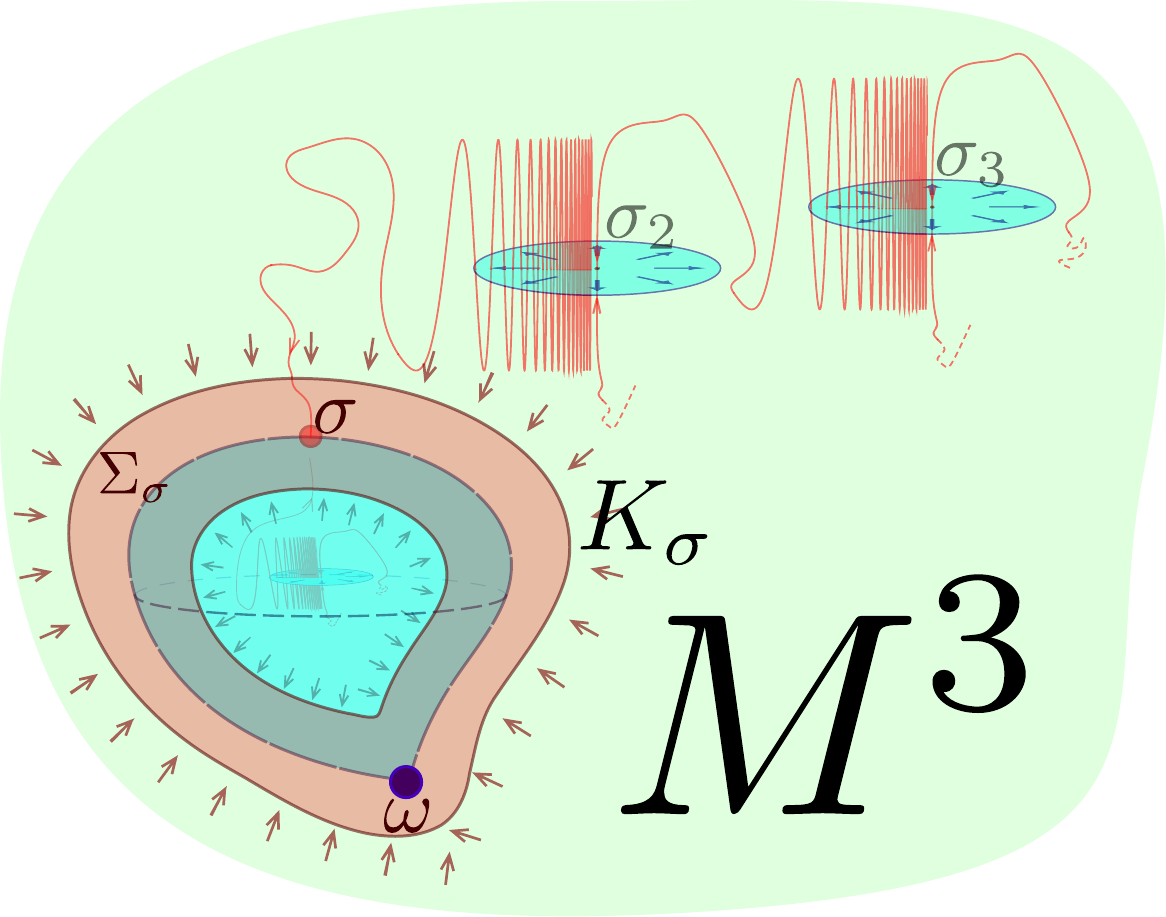}
    \caption{The dynamics of the diffeomorphism $f$ on $M^3$}
    \label{fig:attraction}
\end{figure}

    We denote by $\hat\ell_1,\,\hat\ell_2$ the connected components of the set $\hat \ell_{\Sigma_\sigma}^s$ and by $\eta_{\hat\ell_i}: \pi_1(\hat\ell_i)\to\mathbb Z,\,i=1,2$ the homomorphism induced by the cover $p_{\hat \ell_{\Sigma_\sigma}^s}$. Let us assume $\ell_i=p^{-1}_{\hat \ell_{\Sigma_\sigma}^s}(\hat\ell_i)$. It follows from the definition of the homomorphism $\eta_{\hat\ell_i}$ that it is an isomorphism, hence the set $\ell_i$ is connected. Let $\hat \ell = \ell/a$. Since the point $O$ is a sink of the three-dimensional map $a$, then, by Statement \ref{T_2.1.3.}, $\ell\cong \mathbb S^2 \times \mathbb S^1$ and the homomorphism $\eta_{\hat\ell}:\pi_1(\hat\ell)\to\mathbb Z$ is an isomorphism. Therefore, the manifolds $\hat\ell_i$ and $\hat\ell$ are homeomorphic smooth $3$-manifolds, hence there exists a diffeomorphism $\hat h_i:\hat\ell_i\to\hat\ell_i$ (see \cite{munkres}). Without loss of generality, we assume that $\eta_{\hat \ell} \hat h_i= \eta_{\hat{\ell}_i}$ (otherwise, one may consider its composition with a diffeomorphism $\theta:\mathbb S^2\times\mathbb S^1\to\mathbb S^2\times\mathbb S^1$, given by the formula  $\theta(s,r) = (s, -r)$). 

    By Statement \ref{cycle}, there exists a lift $h_i:\ell_i \to \ell$ of the diffeomorphism $\hat h_i$, smoothly conjugating $\left.f\right|_{\ell _i}$ with $\left.a\right|_{\ell}$.
\end{proof}

Now let $\bar M^\sigma =  \mathbb R^3 \sqcup \mathcal M_\sigma \sqcup \mathbb R^3$,  $M^\sigma =  \mathbb R^3 \cup_{h_1} \mathcal M_\sigma \cup_{h_2} \mathbb R^3$ and let $p_\sigma: \bar M^\sigma \to M^\sigma$ be the natural projection.

\begin{lem}\label{main} The space $M^\sigma$ consists of two connected components $M^{\sigma}_1,\,M^{\sigma}_2$ each of which is a closed smooth $3$-manifold such that $$M^3 = M^\sigma_1 \# M^\sigma_2.$$ Moreover, the manifold $M^{\sigma}_i,\,i=1,2$ admits a diffeomorhism $f_i:M^{\sigma}_i\to M^{\sigma}_i$ belonging to the class $\mathcal G$ and having less saddle points than $f$.
\end{lem}
\begin{proof}   It follows from Lemma \ref{lel} that the manifold $\mathcal M_\sigma$ is a disjoint union of two manifolds, and hence the space $M^\sigma$ has exactly the same number of connected components, let us denote them as $M^{\sigma}_1$ and $M^{\sigma}_2$. 
   Since $h_i$ glues open subsets of $3$-manifolds, then the projection $p_\sigma$ induces the structure of a smooth $3$-manifold on $M^\sigma$. Since the glued manifolds have no boundary, the manifold $M^\sigma$ has no boundary as well.  Due to  the compactness of $M^3$, the manifold $M^\sigma$ is closed. Moreover, it follows directly from the definition of the connected sum that $M^3 = M^\sigma_1 \# M^\sigma_2$.
    
    According to [\cite{Mu2}, Theorem 18.3 (The pasting lemma)], the map $f_\sigma: M^\sigma \to M^\sigma$, defined by the formula
    \begin{align*}
    f_\sigma(x) = 
	\begin{cases}
   p_\sigma(f(p_\sigma^{-1}(x))), \mathit{if }~ x\in p_\sigma(\mathcal 	M^\sigma),\\
   p_\sigma(a(p_\sigma(x))), \mathit{if }~ x \in p_\sigma(\bar{M^\sigma} \setminus \mathcal M_\sigma),
	\end{cases}
    \end{align*}
    is a diffeomorphism. Let $f_i=f_\sigma|_{M_i^\sigma}$(see Fig. \ref{fig:M1}, \ref{fig:M2}).
    \begin{figure}[!ht]
        \begin{minipage}{0.48\textwidth}
        \includegraphics[width = 0.9\textwidth]{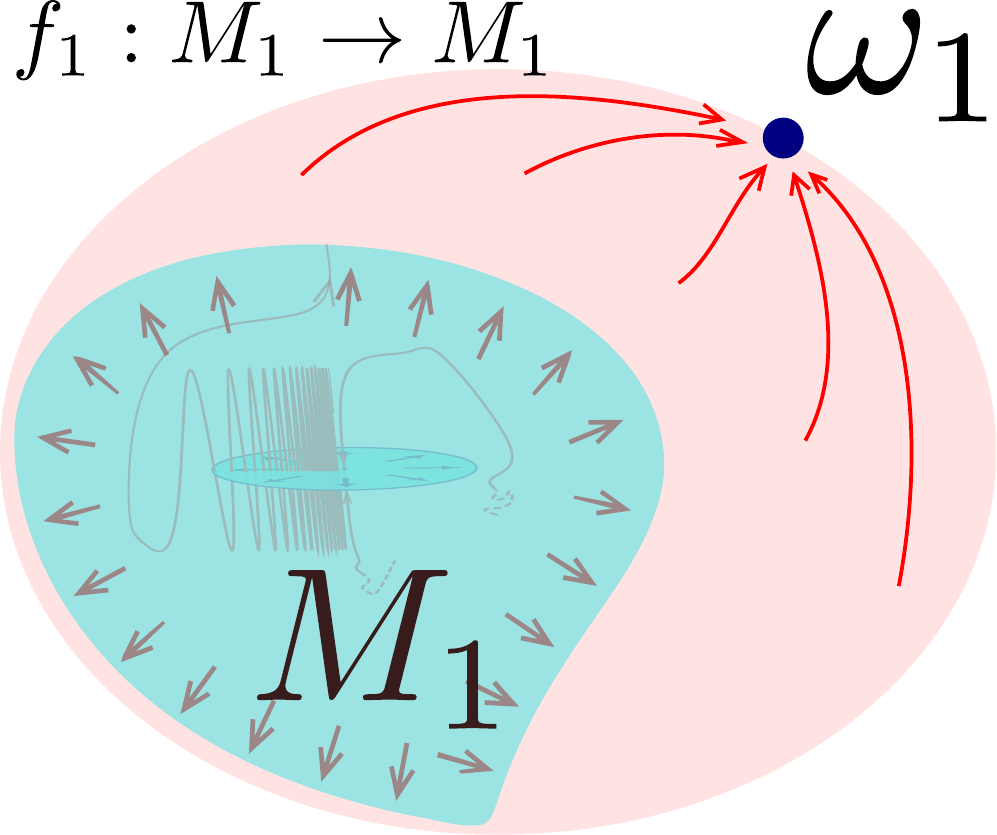}
                \caption{The dynamics of $f_1$ on $M_1^\sigma$}
                \label{fig:M1}
        \end{minipage}
        \hfill
        \begin{minipage}{0.48\textwidth}
        \includegraphics[width = 0.9\textwidth]{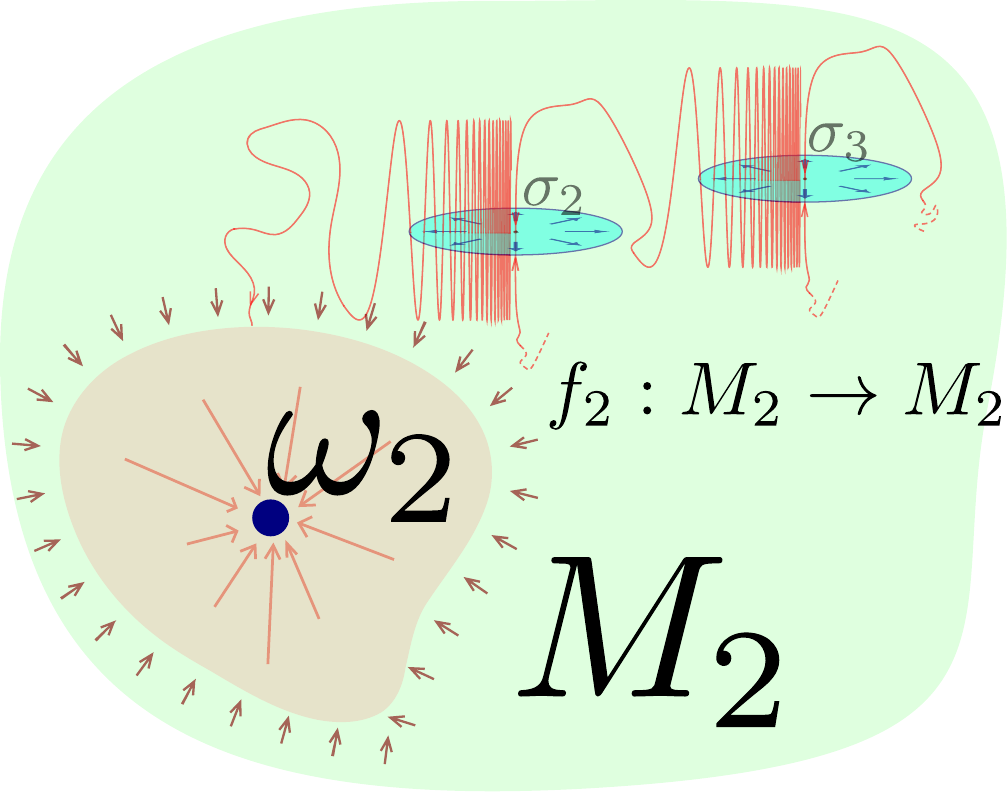}
                \caption{The dynamics of $f_2$ on $M_2^\sigma$}
                \label{fig:M2}
        \end{minipage}
        \end{figure}
     By the construction, the diffeomorphism $f_\sigma$ is smoothly conjugated with $f$ on $p_\sigma(\mathcal 	M^\sigma)$ and with $a$ on $p_\sigma(\bar{M^\sigma} \setminus \mathcal 	M^\sigma)$ ($O_i \,~ (i=1,2)$ is a point  conjugated with the fixed sink $O(0,0,0)$ of $a$). Hence, $f_\sigma\in\mathcal G$ and its non-wandering set have one saddle point less than the non-wandering set of the diffeomorphism $f$.
    \end{proof}

Now let us prove the main result of this paper.

{\bf Theorem 1. } {\it Any closed connected $3$-manifold $M^3$, admitting a diffeomorphism $f \in \mathcal G$, is homeomorphic to the $3$-sphere.} 

\begin{proof} 
    Let $f:M^3 \to M^3$ be from the class $\mathcal G$. Also, we assume that $f$ satisfies the Remark. 
    We prove  Theorem 1 by the  induction on the number $k$ of the saddle points of the diffeomorphism $f$.

    {\bf Base of induction. } $k=0$.
    
 It follows from Statement \ref{NS_diff} that $M^3$ is homeomorphic to the $3$-sphere.  

    {\bf Step of induction. } $k>0$.

    {\bf Inductive hypotheses. } {\it  Any diffeomorphism from the class $\mathcal G$, the number of saddle points in which is less than some natural number $k$, can be defined only on a manifold homeomorphic to the $3$-sphere.}

    The diffeomorphism $f: M^3 \to M^3$ lies in $\mathcal G$ and have exactly $k$ saddle points. Due to  Lemma \ref{sigma}, there exists a saddle $\sigma$ whose the unstable manifold has no heteroclinic intersections. This saddle was chosen according to the order \ref{order}.
     By Lemma \ref{main}, $M^3 = M^\sigma_1 \# M^\sigma_2$ and the manifold $M^{\sigma}_i,\,i=1,2$ admits a diffeomorphism $f_i:M^{\sigma}_i\to M^{\sigma}_i$ from the class $\mathcal G$ which have less saddle points than $f$.

    In this case, it follows from the inductive hypotheses that $M^\sigma_i \cong \mathbb S^3$. Thus, $M^3$ is a connected sum of the $3$-spheres and, consequently, $M^3\cong\mathbb S^3$. 
\end{proof}

\end{document}